\documentclass[10pt]{amsart}
\usepackage{amsmath}
\usepackage{amsfonts}
\usepackage{amsthm}
\usepackage{enumerate}
\usepackage{amssymb}
\usepackage{amscd}

\DeclareMathOperator{\Tor}{Tor}

\DeclareMathOperator{\Ext}{Ext}

\DeclareMathOperator{\im}{Im}

\begin{document}
\theoremstyle{definition}
\newtheorem*{defn}{Definition}
\theoremstyle{remark}
\newtheorem*{rmk}{Remark}
\newtheorem*{rmks}{Remarks}
\theoremstyle{plain}
\newtheorem*{lem}{Lemma}
\newtheorem*{prop}{Proposition}
\newtheorem*{thm}{Theorem}
\newtheorem*{example}{Example}
\newtheorem*{cor}{Corollary}
\newtheorem*{conj}{Conjecture}
\newtheorem*{hyp}{Hypothesis}
\newtheorem*{thrm}{Theorem}
\newtheorem*{quest}{Question}
\theoremstyle{remark}
\newcommand{\Fp}{\mathbb{F}_p}
\newcommand{\Oa}{\mathcal{O}(\alpha)}
\newcommand{\Zp}{\mathbb{Z}_p}
\newcommand{\Qp}{\mathbb{Q}_p}
\newcommand{\GK}{\mathrm{GKdim}}
\newcommand{\invlim}{\lim\limits_{\longleftarrow}}
\newcommand{\injdim}{\mathrm{inj.dim}}
\newcommand{\lmod}[1]{\mathrm{Pmod}(#1)}
\newcommand{\rmod}[1]{\mathrm{Pmod-} #1}

\title{Euler characteristics, Akashi series and compact $p$-adic Lie groups}
\author{Simon Wadsley}
\address{Homerton College, Cambridge, CB2 8PQ}
\maketitle

\section{Introduction}
In this paper following \cite{Coa1999,CoaSuj1999} %, CoaHow2001, Tot1999, CoaSuj1999, CoaSujWin2001} 
we are interested in computing an Euler characteristic for finitely generated modules over Iwasawa algebras. Suppose that $G$ is a compact $p$-adic Lie group without $p$-torsion, and $M$ is a finitely generated $\Zp G$-module. If the homology groups $H_i(G,M)$ are finite for all $i$ then we say that the Euler characteristic is well-defined and takes the value \[ \chi(G,M):=\prod_{i\geqslant 0} |H_i(G,M)|^{(-1)^n}. \] 

The following is result is classical:

\begin{thm} Suppose that $G\cong\Zp^d$.
\begin{enumerate} \item The Euler characteristic of a finitely generated $\Zp G$-module is well-defined if and only if $H_0(G,M)$ is finite.
\item Moreover, $\chi(G,M)=1$ whenever $M$ is a pseudo-null module with well-defined Euler characteristic. \end{enumerate}
\end{thm}
Both parts of this theorem are known not to be true for general compact $p$-adic Lie groups $G$ without $p$-torsion, see \cite{CoaSchSuj2003/2}. 

In this paper we prove that (1) holds if and only if $G$ is finite-by-nilpotent. Both directions depend on the key fact that the set $S:=\Zp G\backslash \ker(\Zp G\rightarrow \Zp)$ is an Ore set in $\Zp G$ precisely if $G$ is finite-by-nilpotent. 

In addition, we prove that part (2) of the theorem holds whenever $G$ is finite-by-nilpotent. To prove this we prove analogous results for the Akashi series of \cite{CoaSchSuj2003/2} and \cite{CFKSV} (see (\ref{Akashi}) for the definition). In particular we show that if $G\cong H\rtimes\Gamma$ then every finitely generated $\Zp G$-module $M$ such that $H_0(H,M)$ is $\Zp\Gamma$-torsion has well-defined Akashi series if and only if $H$ is finite-by-nilpotent and that if $G$ is finite-by-nilpotent then such a module that is also pseudo-null has trivial Akashi series. This time the first part depends on the fact that $T:=\Zp G\backslash \ker(\Zp G\rightarrow \Zp\Gamma)$ is an Ore set in $\Zp G$ precisely if $H$ is finite-by-nilpotent. We prove the second part using the Hochschild-Serre spectral sequence. We could prove part (2) of the Euler characteristics result directly using the same techniques but we do not because it follows immediately from the Akashi series result.

Whenever $M$ is a finitely generated and torsion for a subset of $S$ (resp. $T$) that is an Ore set of $\Zp G$ the Euler characteristic (resp. Akashi series) of $M$ is well-defined. This raises the interesting algebraic question of what the maximal Ore subsets of $S$ and $T$ are for general compact $p$-adic Lie groups $G$ (with $G\cong H\rtimes\Gamma$ for the $T$ case). Since the long exact sequence of $K$-theory used to formulate the main conjecture in \cite{CFKSV} is defined for any such Ore set and the connecting map will remain surjective whenever $G$ is pro-$p$ this question may well also have arithmetic implications as the $\mathcal{M}_H(G)$-conjecture (\cite[Conjecture 5.1]{CFKSV}) could then be weakened without losing the means of defining of a characteristic element that should be related to a $p$-adic $\mathrm{L}$-function.
 
Following work of Serre \cite{Ser1998} in the case $M$ is finite, Ardakov and the author \cite{ArdWad2008} have described the Euler characteristic of any finitely generated $p$-torsion module in terms of a notion of graded Brauer character for $M$ that is supported on the $p$-regular elements of $G$.  It followed from this description that if $d_G(M)<\dim C_G(g)$ for every $p$-regular element of $G$ then $\chi(G,M)$ must be $1$. Here $d_G(M)$ denotes the canonical dimension of $M$ as defined in section \ref{can} and $C_G(g)$ denotes the centraliser of $g$ in $G$.

Totaro had already extended Serre's work in a different direction. Instead of concentrating on $p$-torsion modules he computed Euler characteristics of modules that are finitely generated as $\Zp$-modules. Part of his main result was the following:

\begin{thm}[\cite{Tot1999}, Theorem 0.1]  Let $p$ be any prime number. Let G be a compact $p$-adic Lie group of dimension at least 2, and let $M$ be a finitely generated $\Zp$-module with $G$-action. Suppose that the homology of the Lie algebra $\mathfrak{g}_{\Qp}$ of $G$ acting on $M\otimes\Qp$ is 0; this is equivalent to assuming that the homology of any sufficiently small open subgroup $G_0$ acting on $M$ is finite, so that the Euler characteristic $\chi(G_0, M)$ is defined. Then the Euler characteristics $\chi(G_0, M)$ are the same for all sufficiently small open subgroups $G_0$ of $G$. 

The common value of these Euler characteristics is $1$ if every element of the Lie algebra $\mathfrak{g}_{\Qp}$ has centraliser of dimension at least 2. Otherwise, there is an element of $\mathfrak{g}_{\Qp}$ whose centraliser has dimension $1$, and then the common value is not $1$ for some choice of module $M$. 
\end{thm}

Looking at these results together, and recalling that $d_G(M)\leq 1$ for any finitely generated $\Zp$-module, we might be tempted to make the following conjecture:

\begin{conj} If $G$ is a compact $p$-adic Lie group without $p$-torsion and $M$ is a finitely generated $\Zp G$-module with well-defined Euler characteristic such that $d_G(M)<\dim C_G(g)$ for all $g\in G$ then $\chi(G,M)=1$. 
\end{conj}

In fact our Theorem might indicate that an even stronger conjecture is true. Recall that every compact $p$-adic Lie group has an associated $\Zp$-Lie algebra $\mathfrak{g}$ and each automorphism of $G$ induces an automorphism of $\mathfrak{g}$. In this way for each $g\in G$ conjugation by $g$ induces an element $\theta(g)$ of $GL(\mathfrak{g})$. Let $\mathfrak{g}^0(g)$ denote the generalised eigenspace $\theta(g)$ for the eigenvalue $1$. That is \[ \mathfrak{g}^0(g):=\{x\in\mathfrak{g}|(\theta(g)-1)^n=0\mbox{ for some }n>0\}. \]

\begin{conj} If $G$ is a compact $p$-adic Lie group without $p$-torsion and $M$ is a finitely generated $\Zp G$-module with well-defined Euler characteristic such that $d_G(M)<\dim\mathfrak{g}^0(g)$ for each $g\in G$ then $\chi(G,M)=1$.
\end{conj}

Of course, if either $g$ has finite order or $g$ has infinite order and $\dim C_\mathfrak{g}(g)=1$ then $\mathfrak{g}^0(g)=C_\mathfrak{g}(g)$. Moreover if $G$ is finite-by-nilpotent then $\mathfrak{g}^0(g)=\dim G$ for every $g\in G$.

\subsection{Acknowledgments} Much of this work was done whilst the author was a EPSRC postdoctoral fellow under research grant EP/C527348/1. He would like to thank Konstantin Ardakov for many helpful conversations.

\section{Preliminaries}

\subsection{Notation}

Let $G$ be a compact $p$-adic Lie group. We define the Iwasawa algebra \[ \Zp G:= \invlim\Zp [G/N] \] where $N$ runs over all the open normal subgroups of $G$ and $\Zp[G/N]$ denotes the usual algebraic group algebra. If $H$ is a closed normal subgroup of $G$, we write $I_{H,G}$ for the kernel of the augmentation map $\Zp G\rightarrow \Zp G/H$. 

Given a profinite ring $R$, we write $\lmod{R}$ for the category of profinite left $R$-modules and continuous $R$-module homomorphisms. Then $H_i(G,-)$ is the $i$th derived functor of the functor \[ (-)_G\colon \lmod{\Zp G}\rightarrow\lmod{\Zp} \] that sends a module $M$ to its $G$-coinvariants $M/I_{G,G} M$. Since $(-)_G=\Zp\otimes_{\Zp G}(-)$ as functors, it follows that $H_i(G,M)\cong \Tor_i^{\Zp G}(\Zp,M) \mbox{ for each }i\geqslant 0$ (see \cite[section 6.3]{RibZal2000} for more details.) 
 
If $X$ is a subset of a profinite group we will write $\langle X\rangle$ for the closed subgroup of $G$ generated by $X$. We write $Z(G)$ for the centre of $G$.

If $S$ is a (left and right) Ore set in a ring $R$ and $M$ is an $R$-module, we write $R_S$ for the localisation of $R$ at $S$, and $M_S$ for the localisation of $M$ at $S$.

$\Gamma$ will always denote a group isomorphic to $\Zp$. 
 
\subsection{A little group theory} \label{group}

\begin{defn} Recall that a group $G$ is \emph{finite-by-nilpotent} if it has a finite normal subgroup $N$ such that $G/N$ is nilpotent.
\end{defn}

\begin{lem} Suppose that $G$ is a compact $p$-adic Lie group that is finite-by-nilpotent and has no elements of order $p$
\begin{enumerate}
\item $G$ has a maximal finite normal subgroup $\Delta^+(G)$
\item $G/\Delta^+$ is a nilpotent pro-$p$ group without elements of order $p$.
\end{enumerate}
\end{lem}

\begin{proof}
For part (1) see \cite[1.3]{ArdBro2007}. Part (2) follows from \cite[Lemma 4.1]{Ard2006}.
\end{proof}

\subsection{Properties of $\Zp G$} 

We record some standard properties of Iwasawa algebras that we will use without further comment; see \cite{ArdBro2006} for references to proofs.

\begin{lem} Suppose that $G$ is a compact $p$-adic Lie group of dimension $d$.
\begin{enumerate}
\item $\Zp G$ is a left and right Noetherian ring;
\item $\Zp G$ is an Auslander-Gorenstein ring of dimension $d+1$;
\item $\Zp G$ has finite global dimension if and only if $G$ has no elements of order $p$;
\end{enumerate}
\end{lem}

\subsection{The Canonical dimension function} \label{can}
Recall that if $R$ is an Auslander-Gorenstein ring then there is a canonical dimension function $\delta$ on the category of non-zero finitely generated $R$-modules given by $\delta(M)=\dim R-j_R(M)$ where $j_R(M)=\inf\{j|\Ext^j_R(M,R)\neq 0\}$. 

We say a module is \emph{pseudo-null} if $\delta(M)\leqslant \injdim(R)-2$.

When $R=\Zp G$ for $G$ a compact $p$-adic Lie group we write $d_G(M)$ for $\delta(M)$. Then $M$ is pseudo-null as a $\Zp G$-module when $d_G(M)\leqslant\dim G-1$. 
 
\subsection{Homology and base change} \label{basechange}

\begin{lem} Suppose we have rings $R$ and $S$, a ring homomorphism $R\rightarrow S$, a right $S$-module $N$, and a left $R$-module $M$.
\begin{enumerate}
\item If $S$ is a flat as a right $R$-module then \[ \Tor_i^R(N,M)\cong\Tor_i^S(N,S\otimes_R M) \] for each $i\geqslant 0$.
\item In general, there is a base change spectral sequence \[E^2_{ij}=\Tor_i^S(N,\Tor_j^R(S,M)) \Longrightarrow \Tor_{i+j}^R(N,M).\]
\end{enumerate}
\end{lem}

\begin{proof}
Part (1) follows immediately from part (2). Part (2) is \cite[Theorem 5.6.6]{Wei1995}. 
\end{proof}

\subsection{Computation of homology groups when $G\cong\Zp$} \label{homology}
\begin{lem} If $G=\Zp=\langle z\rangle$ and $M$ is a profinite $\Zp G$-module then 
\begin{enumerate}
\item $H_i(G,M)=0$ unless $i=0,1$;
\item $H_0(G,M)=M/(z-1)M$;
\item $H_1(G,M)=M^G=\ker (z-1)\colon M\rightarrow M$. 
\end{enumerate}
\end{lem}

\begin{proof} The map $\Zp G\rightarrow\Zp G$ sending $\alpha$ to $(z-1)\alpha$ defines a projective resolution of $\Zp$ as a $\Zp G$-module. All three parts follow easily.
\end{proof}

\subsection{Localisation at augmentation ideals} \label{localisable}

Recall that a semi-prime ideal $I$ in a ring $R$ is \emph{localisable} if the set of elements of $r$ in $R$ such that $r+I$ is regular in $R/I$ forms an Ore set in $R$. The following result of Ardakov explains, at least in part, the importance of the condition that $G$ be finite-by-nilpotent.

\begin{thm}[Theorem A of \cite{Ard2006}] If $G$ is a compact $p$-adic Lie group and $H$ is a closed normal subgroup then the kernel $I_{H,G}$ of the augmentation map $\Zp G\rightarrow\Zp G/H$ is localisable if and only if $H$ is a finite-by-nilpotent group.
\end{thm}

\subsection{Akashi series}\label{Akashi}

This material in this section is largely taken from \cite[section 4]{CoaSchSuj2003/2}.

\begin{defn} If $M$ is a finitely generated torsion $\Zp\Gamma$-module then there is an exact sequence of $\Zp\Gamma$-modules \[ 0\rightarrow \bigoplus_{i=1}^r \Zp\Gamma/\Zp\Gamma f_i\rightarrow M\rightarrow D\rightarrow 0 \] with $D$ pseudo-null. The \emph{characteristic element} of $M$ is defined by \[ f_M=\prod f_i\] and is uniquely determined up to multiplication by a unit in $\Zp\Gamma$.
\end{defn}

Now write $Q(\Gamma)$ for the field of fractions of $\Zp\Gamma$. 

\begin{defn} If $G$ is isomorphic to a semidirect product $G\cong H\rtimes\Gamma$ and $M$ is a finitely generated $\Zp G$-module such that $H_j(H,M)$ is a torsion $\Zp\Gamma$-module for each $j\geq 0$, then the \emph{Akashi series} of $M$ is given by \[ Ak_H(M)=\prod_{j\geqslant 0}(f_{H_j(H,M)})^{(-1)^j}\in Q(\Gamma)/(\Zp\Gamma)^\times.\] In this case we say that the Akashi series of $M$ is well-defined. We will supress the subscript $H$ when no confusion will result.
\end{defn}

\begin{rmk} Strictly our definition of the Akashi series is a little more general than that in \cite{CoaSchSuj2003/2}; we extend their definition to some modules that need not be finitely generated over $\Zp H$. One consequence of this is that unlike in their version if a $p$-torsion module has well-defined Akashi series it need not necessarily be trivial. However the proofs in the lemma below are identical. 
\end{rmk}

\begin{lem} Suppose $G$ is isomorphic to a semidirect product $G\cong H\rtimes\Gamma$
\begin{enumerate} 
\item If $0\rightarrow L\rightarrow M\rightarrow N\rightarrow 0$ is a short exact sequence of finitely generated $\Zp G$-modules with well-defined Akashi series then $Ak(M)=Ak(L).Ak(N)$. 
\item If $M$ is a finitely generated $\Zp G$-module with well-defined Euler characteristic, then $M$ has well-defined Akashi series and \[ \chi(G,M)=|\epsilon(Ak(M))|_p^{-1} \] where $\epsilon$ is the augmentation map $\Zp\Gamma\rightarrow \Zp$.
\end{enumerate}
\end{lem}

\section{Characterisation of when the nature of the zeroth homology group suffices to determine well-definition}

\subsection{Modules with well-defined Euler characteristic} \label{Eulerdefd}

Suppose that $G$ is any compact $p$-adic Lie group with no elements of order $p$. We want to study those finitely generated left $\Zp G$-modules $M$ with well-defined Euler characteristic. Since the groups $H_i(G,M)\cong\Tor_i^{\Zp G}(\Zp,M)$ are finitely generated $\Zp$-modules we just need to know when they are all $p$-torsion.

\begin{thm} If $G$ is a compact $p$-adic Lie group without no elements of order $p$, then $G$ is finite-by-nilpotent if and only the following are equivalent for a finitely generated left $\Zp G$-module $M$
\begin{enumerate} \item $M$ has well-defined Euler characteristic; \item $H_0(G,M)$ is finite. \end{enumerate} 
 \end{thm}

\begin{proof} 
Using Theorem \ref{localisable}, we know that $G$ is finite-by-nilpotent if and only if $S:=\Zp G\backslash I_{G,G}$ is an Ore set. 

Suppose first that $S$ is an Ore set so we may form the localisation $\Zp G_S$. Then $\Zp G_S$ is a local ring with maximal ideal $(I_{G,G})_S$ and residue field $\Qp$. 

Since $\Zp$ is a $\Zp G$-bimodule, $\Zp\otimes_{\Zp G}M$ is a left $\Zp G$-module. As we also have $(\Zp\otimes_{\Zp G}P)_S\cong \Qp\otimes_{{\Zp G}_S}P_S$ for every finitely generated projective left $\Zp G$-module $P$, it follows that \[ (\Tor_i^{\Zp G}(\Zp,M))_S\cong\Tor_i^{\Zp G_S}(\Qp,M_S) \mbox{ for each }i\geqslant 0.\] Moreover, since $G$ acts trivially on each group $\Tor_i^{\Zp G}(\Zp,M)$, they are all  $p$-torsion if and only if they are all $S$-torsion, if and only if $\Tor_i^{\Zp G_S}(\Qp,M_S)=0$ for every $i\geqslant 0$. 

As $\Zp G_S$ is local with residue field $\Qp$, and $M_S$ is finitely generated over $\Zp G_S$, Nakayama's Lemma tells us that $\Qp\otimes_{{\Zp G}_S} M_S=0$ if and only if $M_S=0$. Thus the Euler characteristic of $M$ is well-defined  if and only if $M_S=0$ if and only if $\Zp\otimes_{\Zp G}M$ is finite as required. 

Suppose now that $S$ is not an Ore set This means that we can find $r\in\Zp G$ and $s\in S$ such that $Sr\cap \Zp Gs=\emptyset$. We consider the cyclic left $\Zp G$ module $M=\Zp G/\Zp G\langle r,s\rangle$ for this pair $r,s$. 

There is a free resolution of $M$ that begins \[ \cdots\rightarrow(\Zp G)^d\stackrel{d_1}{\rightarrow}(\Zp G)^2\stackrel{d_0}{\rightarrow}\Zp G\rightarrow M\rightarrow 0, \] with $d_0(\alpha,\beta)=\alpha r+\beta s$.

Now applying $\Zp\otimes_{\Zp G}(-)$ to this resolution yields a complex \[ \Zp^d\stackrel{\overline{d_1}}{\rightarrow}\Zp^2\stackrel{\overline{d_0}}{\rightarrow}\Zp\rightarrow 0, \] with homology $H_\bullet(G,M)$.

The condition that $Sr\cap \Zp Gs=\emptyset$ means that if $d_0(\alpha,\beta)=0$ then $\alpha\in I_{G,G}$. Since $S$ is multiplicatively closed it follows also that $\beta\in I_{G,G}$. Thus \[\im(d_1)=\ker(d_0)\subseteq I_{G,G}(\Zp G)^2\] and so $\overline{d_1}=0$. 
 
We also have $\overline{d_0}(a,b)=a\epsilon(r)+b\epsilon(s)$ where $\epsilon$ is the augmentation map $\Zp G\rightarrow \Zp$. As $s\in S$, $\epsilon(s)\neq 0$ and so $\overline{d_0}$ is not the zero map. It follows that  $H_0(G,M)$ is finite and $H_1(G,M)\cong \Zp$.
\end{proof}

\subsection{Multiplicativity of Euler Characteristic} \label{AddEuler}

Because the $S$-torsion modules form an abelian subcategory of all finitely generated $\Zp G$-modules we may prove 

\begin{prop} Suppose that $G$ is a finite-by-nilpotent compact $p$-adic Lie group. If $0\rightarrow L\rightarrow M\rightarrow N\rightarrow 0$ is a short exact sequence of finitely generated left $\Zp G$-modules then $M$ has well-defined Euler characteristic if and only if both $L$ and $N$ have well-defined Euler characteristic. Moreover in this case \[ \chi(G,M)=\chi(G,L)\cdot\chi(G,N).\] \end{prop}

\begin{proof} The first part follows from Proposition \ref{Eulerdefd}. The second part can be read off from the long exact sequence of homology.
\end{proof}

\subsection{Modules with well-defined Akashi series} \label{Akdefd}
There are analogous results for Akashi series, in particular

\begin{thm} Suppose $G$ is isomorphic to a semi-direct product $H\rtimes\Gamma$ and define \[ T:=\Zp G\backslash \ker(\Zp G\rightarrow \Zp\Gamma).\] The Akashi series of $M$ is well-defined for every finitely generated left $\Zp G$-module $M$ such that $H_0(H,M)$ is a torsion $\Zp\Gamma$-module if and only if $T$ is an Ore set in $\Zp G$ if and only if $H$ is finite-by-nilpotent.
\end{thm}

\begin{proof} The proof is nearly identical to that of Theorem \ref{Eulerdefd} so briefly:

Once again, that $T$ is an Ore set if and only if $H$ is finite-by-nilpotent follows immediately from Theorem \ref{localisable}. 

If $T$ is an Ore set then $\Zp G_T$ is a local ring with maximal ideal $(I_{H,G})_T$ and $H_0(H,M)$ is a torsion $\Zp\Gamma$-module if and only if \[ H_0(H,M)_T\cong M\otimes_{\Zp G_T}\Zp G_T/(I_{H,G})_T=0\] if and only if $M_T=0$. Thus $H_j(H,M)_T=0$ for every $j\geq 0$ if and only if $H_0(H,M)_T=0$ and the former holds if and only if $H_0(H,M)$ is $\Zp\Gamma$-torsion.

Conversely if $T$ is not an Ore set then there are elements $r\in\Zp G$ and $t\in T$ such that there are no elements $r'\in\Zp G$ and $t'\in T$ with $r't=t'r$. Thus the kernel of the map \[ (\Zp G)^2\rightarrow \Zp G; (\alpha,\beta)\mapsto \alpha r+\beta t\] is contained in $(I_{H,G})^2$. Using this fact to compute the homology groups $H_j(H,M)$ for $M=\Zp G/\Zp G\langle r,t\rangle$ we obtain $H_0(H,M)$ is $\Zp\Gamma$-torsion but $H_1(H,M)$ is not $\Zp\Gamma$-torsion. 
\end{proof}

\begin{cor} If $G$ is as above and $0\rightarrow L\rightarrow M\rightarrow N\rightarrow 0$ is a short exact sequence of finitely generated $\Zp G$-modules then $M$ has well-defined Akashi series if and only if $L$ and $N$ both have well-defined Akashi series.\hfill $\qed$
\end{cor}

\section{Triviality of Euler characterstic for pseudo-nulls}

\subsection{Reduction to torsion-free nilpotent $G$} \label{reduction}

Our main goal now is to prove the second part of our main result: that if $G$ is a finite-by-nilpotent compact $p$-adic Lie group without $p$-torsion and $M$ is a finitely generated pseudo-null $\Zp G$-module with well-defined Euler characteristic then $\chi(G,M)=1$. 

In fact we prove the apparently stronger result that if $G$ is a finite-by-nilpotent compact $p$-adic Lie group and $H$ is a closed normal subgroup such that $G\cong H\rtimes \Gamma$ then whenever $M$ is a finitely generated pseudo-null $\Zp G$-module with well-defined Akashi series we have $Ak_H(M)=1$. To see that the Euler characteristic version follows from this, observe that (except in the trivial case where $G$ is finite) we may always find such a closed normal subgroup $H$ and apply Lemma \ref{Akashi}(2).

We first reduce to the case that $G$ is nilpotent and pro-$p$.

\begin{lem} Suppose that $G\cong H\rtimes\Gamma$ is a compact $p$-adic Lie group with finite normal subgroup $\Delta$ such that $(|\Delta|,p)=1$. If $M$ is a finitely generated $\Zp G$-module then \[ H_i(H,M)\cong H_i(H/\Delta,M_\Delta)\mbox{ for each }i\geqslant 0,\] as $\Zp\Gamma$-modules and so $Ak_H(M)=Ak_{H/\Delta}(M_\Delta)$ if either is well-defined --- of course $G/\Delta\cong (H/\Delta)\rtimes\Gamma$. 

Moreover, $d_{G/\Delta}(M_\Delta)\leqslant d_G(M)$, and so $M_\Delta$ is a pseudo-null $\Zp G/\Delta$-module if $M$ is a pseudonull $\Zp G$-module.
\end{lem}

\begin{proof} First recall that $\Zp$ with the trivial $\Delta$-action is a projective right $\Zp\Delta$ module as $|\Delta|$ is a unit in $\Zp$. Since the induction functor $(-)\hat{\otimes}_{\Zp\Delta}\Zp H$ from profinite right $\Zp\Delta$-modules to profinite right $\Zp H$-modules is left-adjoint to the restriction functor, it sends projective modules to projective modules and so in particular $\Zp H/\Delta\cong \Zp \hat{\otimes}_{\Zp \Delta} \Zp H$ is a projective profinite right $\Zp H$-module and so flat as a right $\Zp H$-module. 

Using Lemma \ref{basechange}(1) we may conclude that $H_i(H,M)\cong H_i(H/\Delta,H_0(\Delta,M))$ for each $i\geqslant 0$ as $\Zp\Gamma$-modules and the first part follows.

By considering a finitely generated projective resolution of $M$ we can also show \[ \Ext^j_{\Zp G}(M,\Zp G)\otimes_{\Zp \Delta}\Zp\cong \Ext^j_{\Zp G/\Delta}(M_\Delta,\Zp G/\Delta) \] for each $j\geqslant 0$ and the second part follows too as $\dim G=\dim G/\Delta$.
\end{proof}

By applying this Lemma in the case $\Delta=\Delta^+$ is the maximal finite normal subgroup of $G$ and using Lemma \ref{group} we have reduced the calculation of Euler characteristics of finitely generated pseudo-null $\Zp G$-modules when $G$ is finite-by-nilpotent without elements of order $p$ to the case when $G$ is torsion-free nilpotent and pro-$p$. 

\subsection{The torsion-free nilpotent case}\label{final}

Suppose now that $G$ is a finitely generated nilpotent pro-$p$ group without torsion and we have a fixed decomposition $G\cong H\rtimes\Gamma$. 
 
\begin{lem} Suppose that $M$ is a $\Zp G$-module with well-defined Akashi series and there exists $z\in Z(G)\cap H$ such that $Z=\langle z\rangle$ is an isolated subgroup of $G$ acting trivially on $M$. Then $Ak_H(M)=1$. 
\end{lem}

\begin{proof}
First notice that $H_0(H,M)=H_0(H/Z,M)$ is a torsion $\Zp\Gamma$-module, that $G/Z\cong H/Z\rtimes\Gamma$ and that $H/Z$ is also torsion-free and nilpotent so $Ak_{H/Z}(M)$ is also well-defined by Proposition \ref{Eulerdefd}. Now the proof is nearly identical to the proof of \cite[Corollary 12.2]{ArdWad2008} but we'll sketch it here for the sake of the reader.

By Lemma \ref{homology} $H_i(Z,M)$ vanishes for $i>1$ and is isomorphic to $M$ as a left $\Zp G/Z$-module for $i=0,1$ since $Z$ is central in $G$. Thus Lemma \ref{basechange}(2) describes a spectral sequence of $\Zp\Gamma$-modules with second page \[ E_{ij}=H_i(H/Z,H_j(Z,M))\] that is concentrated in rows $j=0,1$. By \cite[Exercise 5.2.2]{Wei1995}, for example, this yields a long exact sequence \[ \cdots\rightarrow H_{n+1}(H,M)\rightarrow H_{n+1}(H/Z,M)\rightarrow H_{n-1}(H/Z,M)\rightarrow H_n(H,M)\rightarrow\cdots \] of $\Zp\Gamma$-modules

The result now follows from the multiplicativity of characteristic elements. \end{proof}

\begin{thm} If $G\cong H\rtimes\Gamma$ is a torsion-free nilpotent $p$-adic Lie group, and $M$ is a pseudo-null $\Zp G$-module with well-defined Akashi series then $Ak_H(M)=1$.
\end{thm}

\begin{proof} Since $M$ is Noetherian and $Ak_H(0)=1$ we may define $N$ to be a maximal submodule of $M$ such that $Ak_H(N)=1$. Using Lemma \ref{Akashi}(1) and Corollary \ref{Akdefd} we see that every non-zero submodule $L$ of $M/N$ satisfies $Ak_H(L)\neq 1$. Thus after replacing $M$ by $M/N$ it suffices to prove that $M$ must have a non-zero submodule $L$ with $Ak_H(L)=1$. 

Next recall (\cite[4.5]{Lev1992}) that whenever we have a finitely generated module $M$ over a Auslander-Gorenstein ring it has a critical submodule; ie a submodule $N$ with the property that every proper quotient has strictly smaller canonical dimension. Using this fact and the remarks in the first paragraph we may assume that our module $M$ is critical.  

Now pick an isolated one-dimensional subgroup $Z$ of $Z(G)\cap H$, then $M^Z$ is a $\Zp G$-submodule of $M$. By the Lemma above $Ak_H(M^Z)=1$, and so we may assume (Lemma \ref{homology}(3)) that $H^1(Z,M)=M^Z=0$. 

Now the homology spectral sequence \[ E^2_{ij}=H_i(H/Z,H_j(Z,M))\Longrightarrow H_{i+j}(H,M) \] has only one non-trivial row on the second page and so $H_i(H,M)=H_i(H/Z,M_Z)$ as $\Zp\Gamma$-modules. Thus $Ak_H(M)=Ak_{H/Z}(M_Z)$. 

But $d_G(M_Z)<d_G(M)$, since $M$ is critical. Also $d_G(M_Z)=d_{G/Z}(M_Z)$ by the Rees Lemma (\cite[Theorem 9.37]{Rotman}, for example). It follows that $M_Z$ is a pseudonull $\Zp G/Z$-module with well-defined Akashi series and by induction on $\dim(G)$ we have $Ak_{H/Z}(M_Z)=1$ --- the result is trivial when $H=1$ by the definition of characteristic element.
\end{proof}

\subsection{Partial converses}

We now discuss some results that put restrictions on the set of groups for which the second conclusion of the main theorem holds.

\begin{defn} We say that \emph{pseudo-nulls are $\chi$-trivial for $G$} if whenever a finitely generated pseudo-null left  $\Zp G$-module $M$ has well-defined Euler characteristic we have $\chi(G,M)=1$.
\end{defn}

\begin{thm} Suppose that $G$ is a compact $p$-adic Lie group without $p$-torsion and pseudo-nulls are $\chi$-trivial for $G$. Then
\begin{enumerate}
\item pseudo-nulls are $\chi$-trivial for every closed subgroup $H$ of $G$;
\item $G$ is $p$-nilpotent, ie. $G/\Delta^+(G)$ is pro-$p$;
\item $\dim C_G(g)>1$ for all $g\in G$ or $G$ has dimension $1$;
\item if $G$ is isomorphic to a semidirect product $\Zp^d\rtimes\Zp$ then $G$ is nilpotent;
%\item if $G$ is soluble then $G$ is nilpotent. 
\item if $G$ is split-reductive then it is abelian.
\end{enumerate}
\end{thm}

\begin{proof}
For part (1), suppose $H$ is a closed subgroup of $G$ and $N$ is any finitely generated left $\Zp H$-module. By Shapiro's Lemma (see \cite[Theorem 6.10.9]{RibZal2000}, for example)  we have $H_i(G,\Zp G\otimes_{\Zp H} N)\cong H_i(H,N)$ for each $i\geqslant 0$, so it suffices to prove that if $N$ is pseudo-null as a $\Zp H$-module then $\Zp G\otimes_{\Zp H}N$ is pseudo-null as a $\Zp G$-module.

By inducing a finitely generated projective resolution of $N$ as a $\Zp H$-module to a projective resolution of $\Zp G\otimes_{\Zp H}N$ as a $\Zp G$-module we see that \[ \Ext^j_{\Zp H}(N,\Zp H)\otimes_{\Zp H}\Zp G\cong \Ext^j_{\Zp G}(\Zp G\otimes_{\Zp H}N,\Zp G)\] for each $j\geqslant 0$ and we are done.

Part (2) follows from \cite[Lemma 7.4 \& Theorem 11.5]{ArdWad2006}. In particular there is a pseudo-null $p$-torsion $\Zp G$-module $M$ with $\chi(G,M)\neq 1$.

For part (3), since all open subgroup are closed, part (1) tells us that it suffices to prove the result for an open subgroup of $G$. The result now follows immediately from Totaro's theorem quoted in the introduction.

For part (4), suppose $G\cong \Zp^d\rtimes \Zp$, let $H$ be the closed normal subgroup of $G$ isomorphic to $\Zp^d$ and let $g$ generate a complement to $H$ in $G$. Now the action of $g$ on $H$ has a minimal polynomial $p(t)$, say and we may write $p(t)=(t-1)^rq(t)$ with $q(t)$ and $(t-1)$ relatively prime. If $q$ is constant then $G$ is nilpotent so we may assume it is not. Let $K=\langle\ker q(g),g\rangle$, a closed subgroup $G$ with $\dim C_K(g)=1$. By part (3) pseudo-nulls are not $\chi$-trivial for $K$ and the result follows by part (2). 
 
To prove part (5) we first notice that parts (1) and (4) imply that it suffices to show that if $G$ is not abelian then it must have a non-nilpotent closed subgroup isomorphic to a semi-direct product $\Zp^d\rtimes \Zp$. If the associated Lie algebra is split-reductive but not abelian then $G$ has a non-abelian subgroup isomorphic to a semi-direct product $\Zp\rtimes\Zp$. \end{proof} 

\begin{rmks} \hfill

\begin{enumerate}
\item It is easy to see that the proof of part (5) applies to a much wider class of groups than non-abelian split-reductive groups. 

\item However, one class of examples that such arguments don't easily apply to includes groups that are four-dimensional and soluble and are isomorphic to a semi-direct product $H_3\rtimes\Zp$ where $H_3$ is a $3$-dimensional Heisenberg pro-$p$ group and the complementary copy of $\Zp$ is generated by an element that induces a automorphism of infinite order on $H_3$ that acts trivially on the centre of $H_3$ and fixes no other one-dimensional subgroup.

\end{enumerate}
\end{rmks}

\bibliographystyle{plain}
\bibliography{../../biblio/references}

\end{document}